\documentclass[11pt]{article}
\usepackage{amsmath, amscd, amssymb, latexsym, epsfig, color, amsthm}
\numberwithin{equation}{section}

\newtheorem{theorem}{Theorem}[section]
\newtheorem{proposition}[theorem]{Proposition}
\newtheorem{question}[theorem]{Question}
\newtheorem{corollary}[theorem]{Corollary}

\newtheorem{lemma}[theorem]{Lemma}

\theoremstyle{definition}
\newtheorem{definition}[theorem]{Definition}
\newtheorem{example}[theorem]{Example}
\newtheorem{remark}[theorem]{Remark}
\DeclareMathOperator{\skel}{Skel}
\DeclareMathOperator\lk{\mathrm{lk}}
\DeclareMathOperator\Star{\mathrm{st}}

\newcommand{\field}{{\bf k}}
\newcommand{\N}{{\mathbb N}}
\newcommand{\I}{{\mathcal I}}
\newcommand{\C}{{\mathcal C}}

\title{From flag complexes to banner complexes}
\author{Steven Klee\\
\small Department of Mathematics \\[-0.8ex]
\small Seattle University\\[-0.8ex]
\small Seattle, WA 98122, USA\\[-0.8ex]
\small \texttt{klees@seattleu.edu}
\and Isabella Novik
\thanks{Research is partially
supported by NSF grant DMS-1069298}\\
\small Department of Mathematics\\[-0.8ex]
\small University of Washington\\[-0.8ex]
\small Seattle, WA 98195-4350, USA\\[-0.8ex]
\small \texttt{novik@math.washington.edu}
}

\begin{document}
\maketitle

\begin{abstract}
A notion of an $i$-banner simplicial complex is introduced.
For various values of $i$, these complexes interpolate between
the class of flag complexes and the class of all simplicial
complexes.  Examples of simplicial spheres of an arbitrary
dimension that are $(i+1)$-banner but not $i$-banner are
constructed. It is shown that several theorems for flag
complexes have appropriate $i$-banner analogues.  Among them are
(1) the codimension-$(i+j-1)$ skeleton of an $i$-banner homology sphere
$\Delta$ is $2(i+j)$-Cohen--Macaulay for all $0\leq j\leq \dim\Delta+1-i$, and
(2) for every $i$-banner simplicial complex $\Delta$ there exists a
balanced complex $\Gamma$ with the same number of vertices as $\Delta$ 
whose face numbers of dimension $i-1$ and higher coincide with those of $\Delta$.
\end{abstract}

\section{Introduction}

Flag complexes (see, for instance, \cite{Ath,ChD,Froh,Gal,NePe}
and also \cite[Chapter III]{St96}) form a fascinating family of
simplicial complexes with many nice results and a lot of open problems.
Very recently Bj\"orner and Vorwerk \cite{BjVor} introduced a class of
{\em banner} simplicial complexes that strictly contains that of flag
complexes. Inspired by their work, we define a notion of
$i$-banner complexes: for various values of
$i$ these complexes interpolate between flag complexes and
Bj\"orner--Vorwerk's banner complexes, and between those complexes
and the class of all simplicial complexes; we then extend several known
theorems for flag complexes to the $i$-banner ones.

One motivation for the Bj\"orner--Vorwerk's paper came from
problems related to connectivity of graphs.
A graph $G$ is called {\em $q$-connected} if it  has at least
$q+1$ vertices and removing an arbitrary subset of at most
$q-1$ vertices from $G$ results in a connected graph. Similarly,
a simplicial complex $\Delta$ is {\em $q$-Cohen--Macaulay}
($q$-CM, for short) if removing
an arbitrary subset of at most $q-1$ vertices from $\Delta$ results
in a CM complex of the same dimension as $\Delta$. In particular,
a graph is $q$-connected if and only if it is $q$-CM, and a 0-dimensional
complex is $q$-CM if and only if it has at least $q$ vertices.

Barnette \cite{Bar} (generalizing a result of Balinski \cite{Bal})
proved that the graph, or 1-skeleton, of a
$(d-1)$-dimensional polyhedral pseudomanifold is $d$-connected.
Athanasiadis \cite{Ath} then verified that the graph of a {\em flag}
$d$-dimensional simplicial pseudomanifold is $2d$-connected.
In the case of CM complexes much more can be said:
it was shown by Fl{\o}ystad \cite{Fl} that the codimension-$j$
skeleton of a CM complex is $(j+1)$-CM, while the authors
and Goff \cite[Theorem 4.1]{GKN} proved that the codimension-$j$
skeleton of a flag 2-CM complex is $2(j+1)$-CM.

Bj\"orner and Vorwerk \cite{BjVor} introduced the class of
{\em banner} simplicial complexes, which properly contains
the class of flag simplicial complexes, and
proved that Athanasiadis's result continues to hold in this
generality; namely, the graph of any $d$-dimensional banner normal
pseudomanifold is $2d$-connected.
Here we define the class of $i$-banner complexes for $i\geq 1$.
We show that the relationship between all these classes is as follows.

\begin{itemize}
\item The class of flag complexes coincides with the classes of $1$- and
$2$-banner complexes.

\item For $2 \leq i \leq d$, the class of $(d-1)$-dimensional
complexes that are $i$-banner is strictly contained in that
of $(i+1)$-banner complexes.

\item The class of $(d-1)$-dimensional complexes that are banner
 in the sense of  Bj\"orner--Vorwerk coincides with
the class of $(d-1)$-banner $(d-1)$-dimensional complexes.

\item All $(d-1)$-dimensional simplicial complexes are
$(d+1)$-banner.
\end{itemize}

We will prove that the codimension-$(i+j-1)$ skeleton
of an $i$-banner homology sphere $\Delta$ is $2(i+j)$-CM
for all $0\leq j\leq \dim\Delta+1-i$ (see Theorem \ref{thm:main}).
This result can thus be considered as an ``interpolation''
between Bj\"orner--Vorwerk's theorem \cite[Theorem 4.4]{BjVor} and the
result on CM-connectivity of skeleta of flag
complexes \cite[Theorem 4.1]{GKN} for homology spheres.
We also establish an analogous result for homology manifolds
(see Corollary \ref{cor:main}).

 Bj\"orner and Vorwerk \cite{BjVor} also  introduce a certain invariant
$b_\Delta$ that for a pseudomanifold $\Delta$ controls the connectivity
of the graph of $\Delta$. Here we define a family of invariants
$\{b_i(\Delta) : i\geq 0\}$ that contains $b_\Delta$ from \cite{BjVor}
as $b_1(\Delta)$. When $\Delta$ is a homology sphere,
we provide lower bounds on CM-connectivity of the skeleta of $\Delta$ in
terms of these statistics (see Theorem \ref{thm:banner-connectivity}).
The $(i=1)$-case of  our result recovers Theorem 1.1 of \cite{BjVor} for
homology spheres.

Our next result concerns the face numbers of $i$-banner complexes.
A conjecture posed by Eckhoff \cite{Eckh88} and Kalai (unpublished)
and solved by Frohmader \cite{Froh} posits that for every flag complex
there exists a balanced complex with the same face numbers. We establish
the following extension of this result (Theorem \ref{thm:face}): for
every $i$-banner complex $\Delta$ there is a balanced complex $\Gamma$ 
on the same number of vertices whose face numbers of dimension $i-1$ 
and higher coincide with those of $\Delta$.

Many of the proofs in this paper are natural extensions
of the proofs of the original results for flag complexes.
However we believe that the notion of $i$-banner complexes
will be useful in the study of simplicial complexes and
 their face numbers providing new ways of
``interpolating'' results/conjectures on all
simplicial complexes to a hierarchy of results/conjectures
on $i$-banner complexes for various values of $i$. We list
some open problems along these lines in the last section.
We remark that another family interpolating
between flag complexes and general simplicial complexes is
that of complexes without large missing faces 
(cf.~Lemma \ref{lemma:missing-faces} below); 
such complexes, and especially their $f$-numbers, were extensively 
studied in \cite{Nevo}.

The rest of this note is organized as follows: in Section 2 we collect
several standard definitions and results pertaining to simplicial complexes.
In Section 3, we define $i$-banner complexes and outline some of their basic
properties that will be useful in later sections. In Section~4 we provide
examples of $(i+1)$-banner spheres that are not $i$-banner. In Section 5 we
discuss CM-connectivity of $i$-banner complexes; then in Section 6 we define
$b_i(\Delta)$ invariants and
establish lower bounds on the CM-connectivity of the skeleta of homology
spheres in terms of these statistics. In Section 7 we study face numbers
of $i$-banner complexes. We close in Section 8 with a few open problems.

\section{Preliminaries}
For the sake of completeness, we collect here several definitions and
results pertaining to simplicial complexes.
An excellent reference to this material is Stanley's book \cite{St96}.

A simplicial complex $\Delta$ on the vertex set $V=V(\Delta)$
is a collection of subsets of $V$ that is closed under inclusion and
contains all singletons $\{i\}$ for $i\in V$. The elements of $\Delta$ are
called its {\em faces}, and the maximal faces under inclusion
are called {\em facets}. A set $F\subseteq V$ is called a
{\em missing face} of $\Delta$ if it is not a face of $\Delta$,
but all of its proper subsets are faces. A simplicial complex $\Delta$ is
{\em flag} if all missing faces of $\Delta$ have size~2.

Let $\Delta$ be a simplicial complex on the vertex set $V$.
For $F\in \Delta$, set $\dim F:= |F|-1$ and
define the \textit{dimension of $\Delta$}, $\dim \Delta$,
as the maximal dimension of its faces.
We denote by $f_j=f_j(\Delta)$, where $-1\leq j \leq \dim\Delta$,
the number of $j$-dimensional faces of $\Delta$ ($j$-faces, for short).
 The \textit{$j$-th skeleton of $\Delta$}, $\skel_j(\Delta)$, is
defined as $\skel_j(\Delta):=\{F\in\Delta \ : \ \dim F \leq j\}$.

A $(d-1)$-dimensional simplicial complex $\Delta$ is called
\textit{balanced} if the graph of $\Delta$ is $d$-colorable. Equivalently,
$\Delta$ is balanced if one can partition the vertex set of
$\Delta$ into $d$ sets $V_1,\ldots, V_d$ in such a way that for every face
$F\in\Delta$ and for every $1\leq k \leq d$, $|F\cap V_k|\leq 1$.

Let $\Delta_1$ and $\Delta_2$ be simplicial complexes
on disjoint vertex sets $V_1$ and $V_2$. Then their \textit{join}
is the following simplicial complex on $V_1\cup V_2$,
\[
\Delta_1\ast\Delta_2:=\{F_1\cup F_2 \ : \ F_1\in \Delta_1,  F_2\in
\Delta_2\}.
\]
The \textit{suspension of $\Delta$}, $\Sigma\Delta$, is the join of
$\Delta$ with a 0-dimensional sphere.

If $\Delta$ is a simplicial complex and $F$ is a face of $\Delta$,
 then the \textit{link} of $F$ in $\Delta$ is
 $\lk_{\Delta} F=\lk F:= \{G\in\Delta \ : \
 F\cup G\in\Delta, \,\, F\cap G =\emptyset \}$.
Also, for a subset $W$ of the vertex set $V$ of $\Delta$,
let $\Delta[W]:= \{F \in \Delta\ : \ F \subseteq W\}$
denote the \textit{restriction} of $\Delta$ to the vertices in $W$
and $\Delta_{-W} :=\{F\in\Delta \ : \ F\subseteq V\setminus W\}$
denote the restriction of $\Delta$ to $V\setminus W$.

Using a result of Reisner \cite{Reis},
we say that a $(d-1)$-dimensional complex $\Delta$ is
{\em Cohen--Macaulay} over $\field$ (CM, for short)
if $\tilde{H}_i(\lk F; \field)=0$ for all $F\in\Delta$
(including the empty face) and all $i<d-|F|-1$.
Here $\field$ is a field and $\tilde{H}_i(-, \field)$
denotes the $i$th reduced simplicial homology with coefficients in $\field$.
If in addition, $\tilde{H}_{d-|F|-1}(\lk F; \field)\cong \field$ for every
$F\in \Delta$, then $\Delta$ is called a {\em homology sphere over $\field$}
(or a \textit{Gorenstein* complex over $\field$}).
We say that $\Delta$ is {\em $q$-CM} if for all $W\subset V$ with $|W|\leq
q-1$, the complex $\Delta_{-W}$ is CM  and has the same dimension as

$\Delta$. 2-CM complexes are also known as doubly CM complexes. Every
simplicial sphere (that is, a simplicial complex whose geometric realization
is homeomorphic to a sphere) is a homology sphere (over any $\field$), and
every homology sphere over $\field$ is doubly CM over $\field$. Moreover,
joins of homology spheres are homology spheres.

A simplicial complex $\Delta$ is called \textit{Buchsbaum over $\field$}
if $\Delta$ is \textit{pure} (that is, all facets
of $\Delta$ have the same dimension) and all vertex links of $\Delta$
are CM over $\field$. As with CM complexes, we say that
$\Delta$ is {\em $q$-Buchsbaum} if for all
$W\subset V$ with $|W|\leq q-1$, the complex
$\Delta_{-W}$ is Buchsbaum  and has the same dimension as $\Delta$.
A Buchsbaum complex all of whose vertex links are homology spheres over
$\field$ is a {\em homology manifold over $\field$}. Every simplicial
manifold is a homology manifold (over any $\field$) and every
homology manifold over $\field$ is 2-Buchsbaum over~$\field$.

\section{Basic properties of $i$-banner complexes}
We are now in a position to define $i$-banner complexes and present
some of their properties.

\begin{definition}
Let $\Delta$ be a $(d-1)$-dimensional simplicial complex
on the vertex set $V(\Delta)$.
\begin{itemize}
\item A subset $T$ of $V(\Delta)$ is called a \textit{clique} if every
two vertices of $T$ form an edge of $\Delta$.
\item A clique  $T \subseteq V(\Delta)$ is \textit{critical}
if $T \setminus \{v\}$ is a face of $\Delta$ for some $v \in T$.
\item For an $1 \leq i \leq d$, we say that $\Delta$ is \textit{$i$-banner}
if every critical clique $T$ of size at least $i+1$ is a face of $\Delta$.
\end{itemize}
\end{definition}

\noindent
When $\dim\Delta=d-1$, saying that $\Delta$ is $(d-1)$-banner amounts to
requiring that $\Delta$ has no critical cliques of size $d$ or $d+1$.
Complexes with this property are precisely the \textit{banner complexes}
in the sense of \cite{BjVor}.

The following two lemmas show that $i$-banner complexes for
various values of $i$ interpolate between the class of flag complexes
and the class of all simplicial complexes.

\begin{lemma}  \label{lemma:i+1}
Every $i$-banner complex is also $(i+1)$-banner; every $(d-1)$-dimensional
simplicial complex is $(d+1)$-banner.
\end{lemma}
\begin{proof}
The first part is immediate from the definition of $i$-banner
complexes, while the second part follows from the observation that
a complex containing a critical clique of size at least $d+2$ must
have dimension at least $d$.
\end{proof}

\begin{lemma} \label{lemma:flag}
The following conditions are equivalent for any simplicial complex $\Delta$.
\begin{enumerate}
\item $\Delta$ is flag,
\item $\Delta$ is $1$-banner,
\item $\Delta$ is $2$-banner.
\end{enumerate}
\end{lemma}
\begin{proof}
First we show that $\Delta$ is $1$-banner if and only if it is $2$-banner.
As any $1$-banner complex is $2$-banner, we only need to show the converse.
Suppose $\Delta$ is $2$-banner, and let $T$ be a critical clique. In order
to show that $\Delta$ is also $1$-banner, we need only consider the case
that $|T| = 2$.  But in this case, a critical clique of size 2 is
tautologically a face of $\Delta$; and hence $\Delta$ is $1$-banner.

Now we show that $\Delta$ is $2$-banner if and only if it is flag.
Suppose first that $\Delta$ is $2$-banner, and let $T$ be a clique
in $\Delta$ of size at least $3$.  We prove that $T$ is a face of
$\Delta$ by induction on $|T|$.  When $|T| = 3$, $T$ is a critical clique,
and hence a face of $\Delta$ by definition.  Next, suppose $|T|>3$,
and let $u$ be a vertex of $T$.  Since $T - \{u\}$ is a clique in
$\Delta$ of size at least $3$, the inductive hypothesis implies that
$T - \{u\}$ is a face of $\Delta$. Thus $T$ is a critical clique of size at
least $3$, and hence a face of $\Delta$. Therefore, any clique in $\Delta$
is a face, and so $\Delta$ is flag.

Conversely, if $\Delta$ is a flag complex, then every clique
(and in particular a critical clique) of size at least $3$ in $\Delta$
is a face of $\Delta$.  Thus $\Delta$ is $2$-banner.
\end{proof}

Recall that if $\Delta$ is flag then so are all the links
of $\Delta$ as well as the suspension of $\Delta$.
For $i$-banner complexes the following analogous statements hold.

\begin{lemma} \label{lemma:link-deletion}
Let $\Delta$ be an $i$-banner simplicial complex with $i\geq 2$.
Then the link of $v$, $\lk_{\Delta}(v)$, is $(i-1)$-banner for
every vertex $v$ of $\Delta$.
Moreover, if $F\subseteq V(\lk_{\Delta}(v))$ is a face of $\Delta$
but not a face of $\lk_{\Delta}(v)$, then $\dim F\leq i-2$.
\end{lemma}
\begin{proof}
To prove the first part,
let $T$ be a critical clique of size at least $i$ in $\lk_{\Delta}(v)$.
This means there is some vertex $u \in T$ for which $T-\{u\}$ is a face of
$\lk_{\Delta}(v)$.  Since $(T - \{u\}) \cup \{v\} = (T \cup \{v\})-\{u\}$
is a face of $\Delta$, $T \cup \{v\}$ is a critical clique in $\Delta$ of
size at least $i+1$.  Thus $T \cup \{v\}$ is a face of $\Delta$,
and hence $T$ is a face of $\lk_{\Delta}(v)$.

For the second part, let $F\subseteq V(\lk_{\Delta}(v))$ be a face of
$\Delta$ of size at least $i$. Since all vertices of $F$ are in the link of
$v$, it follows that $F\cup\{v\}$ is a clique in $\Delta$ of size at least
$i+1$; in fact, since $F$ is a face of $\Delta$, this clique is a critical
clique.
Therefore, $F\cup\{v\}$ is a face of $\Delta$, and we infer that $F$ is a
face of $\lk_{\Delta}(v)$.
\end{proof}

\begin{lemma} \label{lemma:suspension}
Let $\Delta$ be a simplicial complex.  Then $\Delta$ is $i$-banner
if and only if the suspension of $\Delta$, $\Sigma\Delta$, is
$(i+1)$-banner.
\end{lemma}
\begin{proof}
Let $u$ and $u'$ be the suspension vertices of $\Sigma\Delta$.

Suppose first that $\Delta$ is $i$-banner, and let $T$ be a
critical clique in $\Sigma\Delta$ with $|T| \geq i+2$.  By definition,
there is a vertex $v \in T$ such that $T - \{v\}$ is a face of
$\Sigma\Delta$.
We examine three possible cases.  If neither $u$ nor $u'$ belongs to $T$,
then $T$ is a face of $\Delta$ since $\Delta$ is $i$-banner.  If $u$ (or
$u'$) belongs to $T$ and $v = u$, then $T - \{u\}$ is a face of $\Delta$,
and hence $T$ is a face of $\Sigma\Delta$.  Finally, if $u$ (or $u'$)
belongs to $T$ but $v \neq u$, then $T - \{u\}$ is a critical clique
of $\Delta$ of size at least $i+1$.  Thus $T - \{u\}$ is a face of
$\Delta$ and hence $T$ is a face of $\Sigma\Delta$.

Conversely, suppose $\Sigma\Delta$ is $(i+1)$-banner,
and let $T$ be a critical clique of $\Delta$ with $|T| \geq i+1$.
Then $T \cup \{u\}$ is a critical clique of $\Sigma\Delta$,
hence a face of $\Sigma\Delta$, which means $T$ is a face of $\Delta$
as well.
\end{proof}

As the suspension of the boundary complex of an $(i-1)$-simplex shows,
a complex with no missing faces of size larger than $i$
(for $i>2$) need not be $i$-banner; the converse statement, however, does
hold:
\begin{lemma} \label{lemma:missing-faces}
Let $\Delta$ be an $i$-banner complex with $i \geq 2$.
Then $\Delta$ has no missing faces of size larger than $i$.
\end{lemma}
\begin{proof}
As every missing face is a critical clique, it follows that
every missing face in $\Delta$ has size at most $i$.
\end{proof}

\section{Examples}\label{section:examples}
\begin{example}\label{nonpure-example}
To construct a $(d-1)$-dimensional complex that is $(i+1)$-banner, but not
$i$-banner, simply take a $(d-1)$-dimensional simplex and
the $(i-1)$-dimensional skeleton of a simplex of
dimension at least $i$,  and glue these two complexes
along one of the $(i-1)$-dimensional faces.
\end{example}
In this section we present a construction of an $(i+1)$-banner {\bf sphere}
of an arbitrary dimension that is not $i$-banner. We start by
constructing $3$-banner spheres that are not flag.

Recall that the {\em stellar subdivision} of a simplicial complex $\Delta$
at a face $\sigma$ (where $\sigma\in\Delta$ and $\dim \sigma>0$) is the
simplicial complex obtained from $\Delta$ by removing all faces containing
$\sigma$ and adding a new vertex $v_\sigma$, as well as all sets of the form
$\tau\cup\{v_\sigma\}$ where $\tau$ does not contain $\sigma$ but
$\tau\cup \sigma\in\Delta$. Observe that if $\sigma$
and $\tau$ are two faces of $\Delta$ and $\dim \sigma>\dim \tau$, then
$\tau$ is also a face of the subdivision of $\Delta$ at $\sigma$.

\begin{proposition}  \label{prop:3-banner}
Let $T$ be the boundary of a triangle, $S$ the boundary of a
$(d-2)$-dimensional simplex, and let $\Delta:= T\ast S$.
Consider the complex $\tilde{\Delta}$ obtained from $\Delta$ by
subdividing all positive-dimensional faces of $\Delta$ except for
the edges of $T$ (starting from top-dimensional
faces and working toward the lower-dimensional ones).
Then $\tilde{\Delta}$ is a $(d-1)$-dimensional simplicial sphere that
is $3$-banner but not flag.
\end{proposition}
\begin{proof}
That $\tilde{\Delta}$ is a simplicial sphere follows from the
fact that the join of two spheres is a sphere and that stellar
subdivisions do not change the homeomorphism type of a complex.
To see that $\tilde{\Delta}$ is not a flag complex, observe that
the vertices of $T$, which we denote by $x,y,z$,
form a missing face in $\tilde{\Delta}$ of size 3.
Finally, note that the faces of $\tilde{\Delta}$
are in bijection with those chains
in the face poset of $\Delta$ that contain at most 1 element from the list
\[
\{x\}, \{y\}, \{z\}, \{x, y\}, \{x, z\}, \{y, z\},
\]
where a chain of faces
$\sigma_s\supset \sigma_{s-1} \supset \cdots \supset \sigma_1 \supset
\emptyset$
in the face poset of $\Delta$ corresponds to the face
$\{v_{\sigma_s}, \cdots , v_{\sigma_1}\}$ of
$\tilde{\Delta}$ if $\sigma_1$ is not in the above list, and to the face
$\{v_{\sigma_s}, \cdots, v_{\sigma_2}\}\cup \sigma_1$ otherwise.
Hence we conclude that (i) there is no clique in $\tilde{\Delta}$ that
properly contains the clique $C=\{x,y,z\}$, and (ii) every clique $C'\neq C$
in $\tilde{\Delta}$ forms a face of $\tilde{\Delta}$. Thus $C$ is the only
critical clique of $\tilde{\Delta}$ that is not a face, and hence
$\tilde{\Delta}$ is $3$-banner.
\end{proof}

\begin{corollary}
For every $2\leq i \leq d$, there exists a $(d-1)$-dimensional
sphere that is $(i+1)$-banner, but not $i$-banner.
\end{corollary}
\begin{proof}
According to Proposition \ref{prop:3-banner}, there exists a $3$-banner
$(d-i+1)$-dimensional simplicial sphere that is not flag. Suspending this
sphere $(i-2)$ times yields, by Lemma~\ref{lemma:suspension}, a
$(d-1)$-sphere that is $(i+1)$-banner, but not $i$-banner.
\end{proof}

\begin{remark}
Bj\"orner and Vorwerk \cite[Example 3.6]{BjVor} construct a
$3$-dimensional sphere that is banner, and hence $3$-banner,
but not flag. Their construction is different from the one
in Proposition~\ref{prop:3-banner} (for instance, it contains
fewer vertices than the complex in our construction). Suspending
this sphere an appropriate number of times produces another family of
$(d-1)$-dimensional spheres that are $(d-1)$-banner, but not
$(d-2)$-banner.
\end{remark}

\section{Cohen-Macaulay connectivity of skeleta}
We now turn to discussing connectivity of $i$-banner complexes.
In particular, our goal in this section is to prove the following theorem:

\begin{theorem} \label{thm:main}
Let $\Delta$ be a $(d-1)$-dimensional homology sphere
 over a field $\mathbf{k}$. If $\Delta$
is $i$-banner for some $1 \leq i \leq d$,
then $\skel_{d-i-j}(\Delta)$ is $2(i+j)$-CM over $\mathbf{k}$ for
all $0\leq j \leq d-i$.
\end{theorem}

We begin with a lemma. (See Remark \ref{remark:h} for a
stronger statement on the face numbers of $(d-1)$-dimensional
2-CM complexes that are $d$-banner.)

\begin{lemma} \label{lemma:gorenstein-num-verts}
Let $\Delta$ be a $(d-1)$-dimensional homology manifold
 over a field $\mathbf{k}$, and suppose that $\Delta$
is $i$-banner for some $1 \leq i \leq d$.  Then $\Delta$ has at
least $2d$ vertices.
\end{lemma}
\begin{proof}
The proof is by induction on $d$.  The claim holds when
$d=2$ since a $1$-dimensional homology manifold $\Delta$
is a graph that is a cycle or a disjoint
union of cycles, but not a triangle (since $i\leq 2$ and hence
$\Delta$ is flag).
Suppose now that $d \geq 3$.

First we claim that the graph of $\Delta$ is not a clique.
To see this, let $F$ be a facet of $\Delta$ and let $v$ be a vertex
that does not belong to $F$.  Such a vertex exists
since $\Delta$ is not a simplex, and hence has at least $d+1$
vertices.  If the graph of $\Delta$ is a clique, then $F \cup \{v\}$ is a
critical $(d+1)$-clique, and thus a face of $\Delta$.
This contradicts the assumption that $\Delta$ is $(d-1)$-dimensional.

Since the graph of $\Delta$ is not a clique, there exist vertices
$u$ and $u'$ such that $\{u,u'\}$ is not an edge in $\Delta$.
Since the link of $u$ is an $(i-1)$-banner, $(d-2)$-dimensional homology
sphere, it has at least $2(d-1)$ vertices by our inductive
hypothesis.  These vertices, together with $u$ and $u'$ account for
at least $2d$ vertices in $\Delta$.
\end{proof}

We are now ready to prove {\em Theorem \ref{thm:main}}. Our proof follows
the same general outline as the proof of \cite[Theorem 4.1]{GKN}.

\begin{proof} 
We begin by establishing notation that will be used throughout the proof.
We will assume that all homology groups are computed with coefficients in
$\mathbf{k}$, and the field will be suppressed from our notation.
Similarly, when we say that a simplicial complex is CM
(respectively $q$-CM), we mean that it is CM over
$\mathbf{k}$ (resp. $q$-CM over $\mathbf{k}$).  Finally, suppose $\Gamma$
is a subcomplex of $\Delta$, and let $W$ be a subset of $V(\Delta)$
(but not necessarily a subset of the vertices of $\Gamma$).
We write $\Gamma_{-W}$ to denote the restriction of $\Gamma$ to the vertices
in $V(\Gamma) \setminus W$.  We will also make use of the following
observation:
\begin{equation}  \label{eq:lk-skel}
\lk_{(\skel_{k}(\Delta))_{-W}}(F) = (\lk_{\skel_k(\Delta)}(F))_{-W} =
(\skel_{k-|F|}(\lk_{\Delta}(F)))_{-W}.
\end{equation}

The proof is by induction on $d$.  Theorem 4.1 in \cite{GKN}
verifies an analogous result for flag complexes. Hence the result
holds when $i\leq2$; this also implies that it holds when $d=2$.

By Reisner's criterion \cite{Reis}, together with the fact that the
$p$-dimensional homology groups of a simplicial complex are determined
by its $(p+1)$-skeleton, we must establish the following three statements
in order to show that $\skel_{d-i-j}(\Delta)$ is $2(i+j)$-CM.
\begin{enumerate}
\item[(A). ] $\Delta_{ - W}$ is at least $(d-i-j)$-dimensional for any
$W \subseteq V(\Delta)$ with $|W| < 2(i+j)$.
\item[(B). ] $\lk_{\skel_{d-i-j}(\Delta)}(F)$ is $2(i+j)$-CM for
any nonempty face $F \in \Delta$.
\item[(C). ] $\widetilde{H}_r(\Delta_{-W}) = 0$ for all $r < d-i-j$ for
any $W \subseteq V(\Delta)$ with $|W| < 2(i+j)$.
\end{enumerate}

First we prove claim (A).  Let $G$ be a maximal face of $\Delta_{-W}$ and
suppose $|G| \leq d-i-j$.  By Lemma \ref{lemma:gorenstein-num-verts},
$\lk_{\Delta}(G)$ has at least $2(d-|G|)$ vertices.  Since
\[
2(d-|G|) \geq 2(d-(d-i-j)) = 2(i+j) > |W|,
\]
there is a vertex of $\lk_{\Delta}(G)$ that does not belong to $W$.
This contradicts the assumption that $G$ is maximal in $\Delta_{-W}$.

Next, we prove claim (B).  We consider two cases based on $|F|$.
If $|F| < i-2$, then $\lk_{\Delta}(F)$ is $(i-|F|)$-banner
by Lemma \ref{lemma:link-deletion}.  Since $\lk_{\Delta}(F)$ is also a
$(d-|F|-1)$-dimensional homology sphere, our inductive hypothesis implies
that $\lk_{\skel_{d-i-j}(\Delta)}(F) = \skel_{(d-|F|)-i-j}(\lk_{\Delta}(F))$
is $2(i+j)$-CM. On the other hand, if $|F| \geq i-2$, then
Lemmas \ref{lemma:flag} and \ref{lemma:link-deletion} imply that
$\lk_{\Delta}(F)$ is a flag homology sphere.  Thus by \cite[Theorem
4.1]{GKN}, $\lk_{\skel_{d-i-j}(\Delta)}(F)$ is $2(i+j)$-CM.

Finally, we prove claim (C). Again, we must consider two cases based
on whether or not the vertices in $W$ form a clique.  In the case that
the vertices of $W$ do not form a clique, the proof is identical to that
of \cite[Theorem 4.1]{GKN}. (It relies on eq.~(\ref{eq:lk-skel}) and a
simple Mayer--Vietoris argument.) So suppose the vertices in $W$ do form a
clique. We claim that $\widetilde{H}^k(\Delta[W]) = 0$ for any $k \geq i-1$.
If $W$ is not a face of $\Delta$, then it contains no critical cliques
of size at least $i+1$.  This means that no $i$ vertices of $W$ form a face
of $\Delta$, and hence $\Delta[W]$ has dimension at most $i-2$.
Thus, indeed we have $\widetilde{H}^k(\Delta[W]) = 0$ for any $k \geq i-1$.
On the other hand, if $W$ is a face of $\Delta$, then $\Delta[W]$ is a
simplex, and $\widetilde{H}^k(\Delta[W]) = 0$ for all $k$.

Since $\Delta$ is a $(d-1)$-dimensional homology sphere,
Alexander duality implies that
\[
\widetilde{H}_r(\Delta_{-W}) \cong \widetilde{H}^{d-r-2}(\Delta[W])
\quad \mbox{for all } r. \]
If $r < d-i-j$, then
$d-r-2 \geq i+j-1 \geq i-1$, and we conclude that
$\widetilde{H}_r(\Delta_{-W}) = 0$ for all such $r$.
\end{proof}

\begin{corollary} \label{cor:main}
Let $\Delta$ be a $d$-dimensional homology manifold
 over a field $\mathbf{k}$. If $\Delta$
is $(i+1)$-banner for some $1 \leq i \leq d$,
then $\skel_{d+1-i-j}(\Delta)$ is $2(i+j)$-Buchsbaum over $\mathbf{k}$ for
all $0 \leq j \leq d+1-i$. Moreover, if $\Delta$ is also connected, then
$\Delta_{-W}$ is connected for any $W \subseteq V(\Delta)$
with $|W| < 2(i+j)$.
\end{corollary}
\begin{proof}
Since $\Delta$ is a $d$-dimensional homology manifold,
all vertex links of $\Delta$ are $(d-1)$-dimensional
homology spheres. Furthermore, since $\Delta$
is $(i+1)$-banner, all vertex links of $\Delta$ are $i$-banner.
Thus, according to Theorem \ref{thm:main},
 $\skel_{d-(i+j)}(\lk v)$ is $2(i+j)$-CM for every vertex $v$ and
all $0\leq j \leq d-i$. Eq.~(\ref{eq:lk-skel}) then completes the proof
for $0\leq j \leq d-i$. Finally, for $j=d+1-i$, the first part follows from
Lemma \ref{lemma:gorenstein-num-verts}.

To prove the second part, let $W = \{v_1,\ldots,v_k\}$ with $k < 2(i+j)$. We
prove that $\Delta_{-W}$ is connected by induction on $k$.  Since we assumed
that $\Delta$ is connected, the result holds when
$k=0$, and we may assume that $k>0$.  Let $W' = W - \{v_k\}$.  The
decomposition $\Delta_{-W'} = \Delta_{-W} \cup \Star_{\Delta_{-W'}}(v_k)$
gives a Mayer-Vietoris sequence
\begin{displaymath}
\cdots \rightarrow \widetilde{H}_0(\lk_{\Delta_{-W'}}(v_k)) \rightarrow
\widetilde{H}_0(\Delta_{-W}) \oplus
\widetilde{H}_0(\Star_{\Delta_{-W'}}(v_k)) \rightarrow
\widetilde{H}_0(\Delta_{-W'}) \rightarrow 0.
\end{displaymath}
By our inductive hypothesis on $k$, $\Delta_{-W'}$ is connected, and
$\lk_{\Delta_{-W'}}(v_k) = (\lk_{\Delta}(v_k))_{-W'}$ is connected
since $\lk_{\Delta}(v_k)$ is $2(i+j)$-CM.  Thus by exactness,
$\Delta_{-W}$ is connected as well.
\end{proof}

\section{An extension of banner connectivity}
In this section we discuss an extension of Theorem \ref{thm:main}
to arbitrary homology spheres. This requires defining the
following family of invariants (cf., \cite[Definition 5.1]{BjVor}).

\begin{definition} \label{defn:b_i}
Let $\Delta$ be a simplicial complex of dimension $d-1$.
For $0\leq i \leq d-1$, define
\begin{eqnarray*}
&&b_i(\Delta)=b_i :=\\
&&\min\{s \ : \  \lk_{\Delta}(F) \text{ is $(d-s-i)$-banner
 or }  \dim\lk_{\Delta}(F)=i
\text{ for all $F \in \Delta$ with $|F| = s$}\}.
\end{eqnarray*}
\end{definition}

\noindent Thus, $0\leq b_i(\Delta)\leq d-i-1$ and
$b_0\geq b_1\geq \cdots\geq b_{d-1}=0$ (this follows from
Lemma \ref{lemma:i+1}).
 If $\Delta$ is the boundary of a $d$-simplex, then
$b_i=d-i-1$ for all $i$. On the other hand, $b_i(\Delta)=0$
if and only if $i=d-1$ or $\Delta$ is $(d-i)$-banner.
We also note that $b_1(\Delta)$ coincides with
the invariant $b_\Delta$ of \cite[Definition 5.1]{BjVor}.

Below is the main result of this section. For
$i=1$ it reduces to \cite[Theorem 1.1]{BjVor} for the
case of homology spheres.

\begin{theorem}\label{thm:banner-connectivity}
Let $\Delta$ be a $(d-1)$-dimensional homology sphere over a
field $\mathbf{k}$. Then
$\skel_i(\Delta)$ is $(2(d-i)-b_i)$-CM over $\mathbf{k}$
for all $0\leq i \leq d-1$.
Moreover, if for a certain $i$, $b_i<d-i-1$, then
$\skel_k(\Delta)$ is $(2(d-k)-b_i)$-CM over $\mathbf{k}$
for all $0\leq k \leq i$.
\end{theorem}

The proof of Theorem \ref{thm:banner-connectivity}
relies on the following lemma.

\begin{lemma}\label{lemma:CM-skeleta}
Let $\Delta$ be a Cohen-Macaulay complex of dimension $d-1$,
and suppose that $\lk_{\Delta}(v)$ is $q$-CM for every vertex
$v \in \Delta$. Then $\skel_{d-2}(\Delta)$ is $(q+1)$-CM.
\end{lemma}
\begin{proof}
Let $W = \{v_1,\ldots,v_k\} \subseteq V(\Delta)$ be a collection
of vertices with $k \leq q$. As in the proof of Theorem \ref{thm:main},
in order to verify that $\skel_{d-2}(\Delta)$
is $(q+1)$-CM, we need to show that
(A) each face of $\Delta_{-W}$ is contained in a face of
dimension at least $d-2$, and
(B) for each face $F$ of $\Delta_{-W}$ of dimension at most $d-2$,
$\widetilde{H}_i(\lk_{\Delta_{-W}}(F)) = 0$ for all $i < d-|F|-2$.
Both of these claims hold when $k=0$ since $\Delta$ is Cohen-Macaulay.
We thus assume that $k>0$ and proceed by induction on $k$.
Let $W':=W - \{v_k\}$.

To prove claim (A), suppose there is a maximal face $\sigma$ of
$\Delta_{-W}$ with $|\sigma| \leq d-2$. Since $\sigma$ is also a
face of $\Delta_{-W'}$, there is a face $\tau$ of $\Delta_{-W'}$
such that $\sigma \subseteq \tau$ and $|\tau| \geq d-1$.
This means that $\tau = \sigma \cup \{v_k\}$ by our assumption
that $\sigma$ is a maximal face of $\Delta_{-W}$,
and hence $\sigma$ is a face of
$\lk_{\Delta_{-W'}}(v_k) = (\lk_{\Delta}(v_k))_{-W'}$.
Since $|W'| < q$ and $\lk_{\Delta}(v_k)$ is $q$-CM,
$(\lk_{\Delta}(v_k))_{-W'}$ is pure of dimension $d-2$.
Thus there is a face $\tau'$ of $\lk_{\Delta}(v_k)$ with
$|\tau'| \geq d-1$ such that $\tau' \supseteq \sigma$.
This contradicts our assumption that $\sigma$ is maximal.

To prove claim (B), let $F$ be a face of $\Delta_{-W}$ of
dimension at most $d-2$, and let $i < d-|F|-2$.
For ease of notation, we define
\[
\Gamma:=\lk_{\Delta_{-W}}(F) = (\lk_{\Delta}(F))_{-W},
\qquad \text{and similarly} \qquad
\Gamma':=(\lk_{\Delta}(F))_{-W'}.
\]
For any $i < d-|F|-2$, an appropriate piece of the
Mayer-Vietoris sequence for the decomposition of
$\Gamma'$ as  $\Gamma' = \Gamma \cup \Star_{\Gamma'}(v_k)$ is
\begin{displaymath}
\cdots \rightarrow  \widetilde{H}_{i}(\lk_{\Gamma'}(v_k)) \rightarrow
\widetilde{H}_i(\Gamma) \oplus \widetilde{H}_i(\Star_{\Gamma'}(v_k))
\rightarrow \widetilde{H}_i(\Gamma') \rightarrow \cdots .
\end{displaymath}
By our inductive hypothesis on $k$, we see that
$\widetilde{H}_i(\Gamma') = 0$.
Since $\lk_{\Gamma'}(v_k)) = (\lk_{\lk_{\Delta}(F)}(v_k))_{-W'}$
and $\lk_{\Delta}(v_k)$ is $(d-2)$-dimensional and $q$-CM,
we obtain that $\widetilde{H}_i(\lk_{\Gamma'}(v_k)) = 0$.
Thus it follows from exactness that $\widetilde{H}_i(\Gamma) = 0$
as well.
\end{proof}

\begin{corollary}\label{cor:cm-skeleta-links}
Let $\Delta$ be a Cohen--Macaulay complex, and suppose that
$\skel_k(\lk_{\Delta}(F))$ is $q$-CM for all faces $F \in \Delta$ with
$|F|=s$. Then $\skel_k(\Delta)$ is $(q+s)$-CM.
\end{corollary}
\begin{proof}
It follows from Reisner's criterion and the assumption that
$\Delta$ is Cohen--Macaulay that $\skel_{k+1}(\Delta)$ is
Cohen--Macaulay as well.  Furthermore, since $\skel_k(\lk_{\Delta}(F))$
is $q$-CM for all faces $F \in \Delta$ with $|F| = s$,
Lemma \ref{lemma:CM-skeleta} implies that $\skel_k(\lk_{\Delta}(G))$
is $(q+1)$-CM for all faces $G \in \Delta$ with $|G| = s-1$.
Inductively, this gives the desired result.
\end{proof}

Now we are ready to prove Theorem \ref{thm:banner-connectivity}

\begin{proof}
Fix $i$ and let $s = b_i(\Delta)$. If $s=d-i-1$, then we only claim that
$\skel_i(\Delta)$ is $(d-i+1)$-CM, and this is immediate from
the fact that $\Delta$ is 2-CM and a result of Fl{\o}ystad \cite{Fl}
asserting that the codimension-1 skeleton of a $q$-CM complex
is $(q+1)$-CM.
Hence assume that $s<d-i-1$. Then by Definition \ref{defn:b_i},
for each face $F \in \Delta$ with $|F| = s$, the link
$\lk_{\Delta}(F)$ is a $(d-s-1)$-dimensional homology sphere
that is $(d-s-i)$-banner.  Thus by Theorem \ref{thm:main},
$\skel_k(\lk_{\Delta}(F))$ is $2(d-s-k)$-CM
provided $0 \leq i-k \leq i$ (which is true by our assumptions).
Therefore by Corollary \ref{cor:cm-skeleta-links},
$\skel_k(\Delta)$ is $(2(d-s-k)+s) = (2(d-k)-b_i(\Delta))$-CM.
\end{proof}

\section{Face numbers of $i$-banner complexes}
Here we discuss how being $i$-banner affects the face
numbers of a simplicial complex. The main result of this section is
the following extension of Frohmader's theorem \cite{Froh}
on the face numbers of flag complexes to $i$-banner complexes.

\begin{theorem} \label{thm:face}
Let $\Delta$ be an $i$-banner simplicial complex (for some $2\leq i \leq d$). 
Then there exists a balanced complex $\Gamma$ with the same number of vertices as $\Delta$ 
such that $f_{k-1}(\Delta)=f_{k-1}(\Gamma)$ for all $k\geq i$.
\end{theorem}

The face numbers of balanced complexes (both numerically and
combinatorially) were characterized in \cite{FFK}. For our purposes we will
only need a combinatorial characterization. It relies on the notion of the
{\em reverse-lexicographic} (``rev-lex'', for short) order: if $A$ and $B$
are two finite equal-size subsets of $\N$---the set of positive integers,
then we say that $A$ precedes $B$ in the rev-lex order
and write $A\prec B$ if $\max\left((A \setminus B) \cup (B \setminus A)\right)$
is an element of $B$. For instance, $\{3,5,7\}\prec \{2,6,7\}$.

We say that a simplicial complex $\Delta$ is {\em $d$-colorable}
if $V(\Delta)$ can be partitioned into $d$ sets $V_1,\ldots,V_d$
(``colors'') in such a way that $|F\cap V_j|\leq 1$ for all
$1\leq j \leq d$. (Thus a $(d-1)$-dimensional complex is $d$-colorable
if and only if it is a balanced complex.)
Also call a subset $A$ of $\N$ {\em $d$-permissible}
if no two distinct elements of $A$ have the same remainder modulo $d$.
Given $k,m\geq 0$, let $\I^d_{k}(m)$ denote the collection of first $m$-many
$d$-permissible $(k+1)$-subsets of $\N$ in rev-lex order, and let
$\C^d_k(m)$ denote the $k$-dimensional pure simplicial complex whose set of
facets is $\I^d_{k}(m)$. Note that $\C^d_k(m)$ is a $d$-colored complex, as
is any complex all of whose faces are $d$-permissible sets
(with one color for each remainder modulo $d$).

\begin{theorem} \label{thm:FFK}
{\rm (Frankl-F\"uredi-Kalai, \cite{FFK})}
Let $a$, $b$, and $k$ be nonnegative integers.
Then there exists a $d$-colorable simplicial complex $\Delta$
with $f_{k-1}(\Delta)=b$ and $f_{k}(\Delta)=a$ if and only if
$\C^d_{k-1}(b) \cup \I^d_k(a)$ is a simplicial complex.
\end{theorem}

We start the proof of Theorem \ref{thm:face} by verifying two lemmas.
The proof of the first of them is essentially the same as Frohmader's 
proof \cite{Froh} but relies on the observation that restrictions of $i$-banner
complexes are $i$-banner and on Lemma \ref{lemma:link-deletion}
instead of analogous statements for flag complexes.

\begin{lemma}  \label{lemma:k}
Let $\Delta$ be a $(d-1)$-dimensional $i$-banner complex and
$i\leq k\leq d-1$. Then there exists a $d$-colorable simplicial
complex $\Gamma$ such that $f_{k-1}(\Gamma)=f_{k-1}(\Delta)$
and $f_{k}(\Gamma)=f_{k}(\Delta)$.
\end{lemma}
\begin{proof} The proof is by induction on $i$. Lemma 4.1 in \cite{Froh}
proves an analogous statement for flag complexes, and hence the result
holds when $i\leq 2$. So fix $i\geq 3$ and $k\geq i\geq 3$
and assume that the statement holds for $i-1$ and all $k'\geq i-1$.

Let $v_0$ be the vertex of $\Delta$ with the property that
$f_{k-1}(\lk_\Delta v_0)\geq f_{k-1}(\lk_\Delta v)$ for all $v\in V$.
In other words, $v_0$ is contained in the most $k$-faces of $\Delta$.
Let $v_1,\ldots, v_s$ be all vertices of $\Delta$ that do not belong
to the link of $v_0$. Set $W_0=\emptyset$, and
for $1\leq j\leq s+1$ consider
$W_j:=\{v_0,v_1,\ldots, v_{j-1}\}$. Thus $\Delta_{-W_{s+1}}$ is the
restriction of $\Delta$ to the vertex set of $\lk_\Delta v_0$, and
\[\Delta=\Delta_{-W_0}
\supseteq \Delta_{-W_1}\supseteq \cdots \supseteq \Delta_{-W_{s+1}}.
\]
Note also that $\Delta_{-W_{j}}\setminus \Delta_{-W_{j+1}}$ is
precisely the set of faces of $\Delta_{-W_{j}}$ that contain $v_j$.

Let $f_{k-1}(\lk_{\Delta_{-W_j}}v_j)=a_j$ and
$f_{k-2}(\lk_{\Delta_{-W_j}}v_j)=b_j$. Two observations are in order.
First, since restrictions of $i$-banner
complexes are also $i$-banner, Lemma \ref{lemma:link-deletion}
implies that for all $0\leq j \leq s$, $\lk_{\Delta_{-W_j}}v_j$
is a $(d-2)$-dimensional $(i-1)$-banner complex. Thus by our induction
hypothesis along with Theorem~\ref{thm:FFK},
\begin{eqnarray}
\Lambda(b_j,a_j)&:=&\C^{d-1}_{k-2}(b_j)\cup \I^{d-1}_{k-1}(a_j) \;\;
\text{is a simplicial complex for all $0\leq j \leq s$, and}
\label{simpl-compl2} \\
\Lambda_0&:=&\C^{d-1}_{k-1}(a_0)
\cup \I^{d-1}_k(f_{k}(\lk_\Delta v_0))
\;\; \text{is a simplicial complex.}
 \label{simpl-compl1}
\end{eqnarray}
Second, since the $p$-faces of $\Delta_{-W_j}$ containing $v_j$ are in
bijection with $(p-1)$-faces of $\lk_{\Delta_{-W_j}}v_j$
and since by Lemma \ref{lemma:link-deletion}, the set
of $(k-1)$- and $k$-faces of $\lk_{\Delta}(v_{0})$ coincides
with the set of $(k-1)$- and $k$-faces of $\Delta_{-W_{s+1}}$,
we conclude that
\begin{equation}  \label{eq:face-numbers}
f_{k}(\Delta)= f_k(\lk_\Delta(v_0)) + \sum_{j=0}^s a_j
\quad \mbox{and} \quad
f_{k-1}(\Delta)=a_0 + \sum_{j=0}^s b_j.
\end{equation}

We now construct a $d$-colorable complex $\Gamma$ with
$f_{k-1}(\Gamma)=f_{k-1}(\Delta)$ and $f_{k}(\Gamma)=f_{k}(\Delta)$.
To do so, start with a $(d-1)$-colorable complex $\Lambda_0$
of eq.~(\ref{simpl-compl1}). By adding to $\Lambda_0$ faces of dimension
$\leq k-2$, if needed, we obtain a $(d-1)$-colorable simplicial complex
$\Lambda$ with
\begin{equation} \label{eq:Lambda-faces}
 f_k(\Lambda)= f_k(\lk_\Delta(v_0)), \;\; f_{k-1}(\Lambda)=a_0, \;\;
\mbox{and } \;\;
\Lambda\supseteq \C^{d-1}_{k-2}(\max\{b_0,b_1,\ldots, b_s\}).
\end{equation}
In addition, by our choice of $v_0$,
\[a_0=f_{k-1}(\lk_\Delta v_0)
\geq f_{k-1}(\lk_\Delta v_j)
\geq f_{k-1}(\lk_{\Delta_{-W_j}} v_j)=a_j \quad \mbox{for all $0\leq j \leq
s$}.
\] Hence,
\[
\Lambda\supseteq \Lambda_0\supset  \I^{d-1}_{k-1}(a_0)\supseteq
\I^{d-1}_{k-1}(a_j),
\]
which combined with eq.~(\ref{eq:Lambda-faces}) yields that
each complex $\Lambda(b_j,a_j)$ of eq.~(\ref{simpl-compl2}) is
a subcomplex of $\Lambda$.
Let $v'_0,v'_1,\ldots,v'_s$ be $s+1$ new vertices of color $d$, and define
\[\Gamma:=\left( \bigcup v'_j\ast \Lambda(b_j,a_j) \right)\cup\Lambda.\]
The above discussion shows that $\Gamma$ is a well defined simplicial
complex, it is $d$-colorable and has the same number of $k$-faces and
$(k-1)$-faces as $\Delta$ (the last statement is a consequence of our
definition of $\Gamma$ and equations (\ref{eq:face-numbers}) and
(\ref{eq:Lambda-faces})). The lemma follows.
\end{proof}

For the second lemma, let $\Gamma(n,d)$ denote the complete multipartite 
graph on $n$ vertices and $d$ parts, each of which has size either 
$\lceil \frac{n}{d} \rceil$ or $\lfloor \frac{n}{d} \rfloor$, and let 
${n \choose j}_d$ denote the number of $j$-cliques in $\Gamma(n,d)$.

\begin{lemma} \label{f-nums-bound-i-banner}
Let $\Delta$ be an $i$-banner simplicial complex of dimension $d-1$ on $n$ vertices.  
Then $f_{i-1}(\Delta) \leq {n \choose i}_d$.
\end{lemma}

\begin{proof}
We prove the claim by induction on $d$ and on $n$.  The result holds when $d=2$ 
by Tur\'an's Theorem and is clear when $n=d$.  So inductively we may 
suppose that the claim holds for all $(i-1)$-banner complexes of dimension $d-2$ 
and for all $i$-banner complexes of dimension at most $d-1$ on fewer than $n$ vertices.  
We will make use of the fact that, 
$f_{i-1}(\Delta) = f_{i-2}(\lk_{\Delta}(u)) + f_{d-1}(\Delta-u)$ for any vertex 
$u \in \Delta$ and use induciton to bound each of these pieces.

Let $F$ be a $(d-1)$-dimensional face of $\Delta$, and let $W \subseteq V(\Delta)$ 
denote the collection of vertices that do not lie on $F$.  
For each vertex $w \in W$, there is at least one vertex $v \in F$ such that
$\{v,w\} \notin \Delta$: otherwise, $F \cup \{w\}$ would be a critical clique 
of size $d+1$ in $\Delta$.   By the pigeonhole principle, there is some vertex 
$v \in F$ that contributes to at least 
$\lceil \frac{n-d}{d}\rceil = \lceil \frac{n}{d} \rceil-1$ missing edges in $\Delta$.  
Since $v$ lies on the $(d-1)$-face $F$, $\lk_{\Delta}(v)$ is an $(d-2)$-dimensional and 
$(i-1)$-banner simplicial complex on at most 
$(n-1)-(\lceil \frac{n}{d}\rceil-1) = n-\lceil\frac{n}{d}\rceil$ vertices.  

On the other hand, $\Delta-v$ is either $(d-1)$-dimensional or $(d-2)$-dimensional.  
In the former case, the inductive hypothesis implies that 
$f_{i-1}(\Delta-v) \leq {n-1 \choose i}_d$; however, in the latter case 
the inductive hypothesis only gives $f_{i-1}(\Delta-v) \leq {n-1 \choose i}_{d-1}$.  
An extension of Tur\'an's theorem due to Zykov \cite{Zyk} shows that $\Gamma(n-1,d)$ 
has the maximum number of $j$-cliques among all $d$-colorable graphs for any 
$1 \leq j \leq d$.  Since $\Gamma(n-1,d-1)$ is $(d-1)$-colorable 
(and hence $d$-colorable), it follows that 
${n-1 \choose i}_{d-1} \leq {n-1 \choose i}_d$ in this case as well.

Thus 
\begin{eqnarray*}
f_{i-1}(\Delta) &=& f_{i-2}(\lk_{\Delta}(v)) + f_{i-1}(\Delta-v) \\
&\leq& {n-\lceil \frac{n}{d} \rceil \choose i}_{d-1} + {n-1 \choose i}_d.
\end{eqnarray*}
In order to complete the proof, we claim that 
${n-\lceil \frac{n}{d} \rceil \choose i}_{d-1} + {n-1 \choose i}_d 
= {n \choose i}_d$.  
Indeed, let $W'$ be a partite set of $\Gamma(n,d)$ with 
$\lceil \frac{n}{d} \rceil$ vertices, and let $x$ be a vertex in $W'$.  
Then there are ${n-\lceil \frac{n}{d} \rceil \choose i}_{d-1}$ $i$-cliques 
in $\Gamma(n,d)$ that contain $x$ and ${n-1 \choose i}_d$ $i$-cliques that 
do not contain $x$. 
\end{proof}

We are now ready to prove Theorem \ref{thm:face}.

\begin{proof} [of Theorem \ref{thm:face}]  \quad If $\Delta$
is an $i$-banner $(d-1)$-dimensional complex on $n$ vertices, 
then by Lemma~\ref{f-nums-bound-i-banner} the complex 
$\C^d_{i-1}(f_{i-1}(\Delta))$ has at most $n$ vertices.  
By Lemma~\ref{lemma:k} and Theorem~\ref{thm:FFK}, the complex
$\C^d_{i-1}(f_{i-1}(\Delta))\cup (\cup_{j=i}^{d-1}\I^d_{j}(f_j(\Delta)))$
is a balanced complex whose face numbers of dimension $i-1$ and higher 
coincide with those of $\Delta$. 
\end{proof}

\begin{remark}  \label{remark:h}
The $h$-numbers, $h_j=h_j(\Delta)$, of a $(d-1)$-dimensional
simplicial complex $\Delta$ are defined by
\[
\sum_{j=0}^d f_{j-1} (x-1)^{d-j}=\sum_{j=0}^d h_j(\Delta)x^{d-j}.
\]
Athanasiadis \cite{Ath} shows that the $h$-numbers of an arbitrary
2-CM flag complex of dimension $d-1$ satisfy $h_j\geq \binom{d}{j}$
for all $0\leq j\leq d$. The proof of this result is an induction
argument that relies on standard results about the $h$-numbers of CM
complexes along with the fact that the graph of a 2-CM flag
complex is not a clique and that vertex links of a flag 2-CM complex
are also flag 2-CM complexes of a smaller dimension.
As the graph of a 2-CM $d$-banner $(d-1)$-dimensional complex
is not a clique (indeed, the only $(d-1)$-dimensional $d$-banner complex
whose graph is a clique is a simplex), and since vertex links of
$d$-banner complexes are $(d-1)$-banner, exactly the same argument
as in \cite{Ath} yields the following more general result:
{\em If $\Delta$ is a 2-CM $(d-1)$-dimensional simplicial
complex that is $d$-banner, then $h_j(\Delta)\geq \binom{d}{j}$
for all $0\leq j\leq d$.}
\end{remark}

\section{Open problems}
We conclude the paper with a few open problems.

As an extension of the Kalai-Eckhoff conjecture,
Kalai also conjectured that the $f$-vector of a Cohen--Macaulay
flag complex is the $f$-vector of a Cohen--Macaulay balanced complex.
It would be interesting to try to extend Theorem \ref{thm:face} to
Cohen--Macaulay complexes.

\begin{question}
Let $\Delta$ be a Cohen--Macaulay $i$-banner simplicial complex.
Is there a Cohen--Macaulay balanced complex $\Gamma$ for which
$f_{k-1}(\Delta) = f_{k-1}(\Gamma)$ for all $k \geq i$?
\end{question}

Let $\Delta$ be an $i$-banner complex.
Theorem \ref{thm:face} together with \cite{FFK} provides an upper
bound on $f_{j}(\Delta)$ in terms of $f_{j-1}(\Delta)$ for all $j\geq i$.
However, at present we do not have any non-trivial bounds on the
lower-dimensional face numbers of $i$-banner complexes.

A simplicial complex $\Gamma$ is called $(a_1,\ldots, a_r)$-balanced,
where $a_1,\ldots a_k$ are positive integers satisfying
$\dim\Delta=\sum_{k=1}^r a_k$,
if the vertex set of $\Delta$ can be partitioned into $r$ sets
$V_1,\ldots,V_r$ in such a way that
\[
|F\cap V_k|\leq a_k \quad
\text{for every face $F$ of $\Delta$ and all $1\leq k \leq r$}.
\]
Thus $(1,1,\ldots,1)$-balanced complexes are the usual balanced complexes;
on the other extreme, every $(d-1)$-dimensional simplicial complex is a
$(d)$-balanced complex. If $a_2=a_3=\cdots=a_r=1$, we write
$(a_1, 1^{r-1})$ instead of $(a_1,a_2,\ldots,a_r)$.
The following question seems to be a natural interpolation between these two
extremes.

\begin{question} \label{quest2}
Let $\Delta$ be an $i$-banner $(d-1)$-dimensional complex. Is there always
an $(i-1,1^{d-i+1})$-balanced complex $\Gamma$ such that
$f_j(\Gamma)=f_j(\Delta)$ for all $j\geq 0$?
\end{question}

The non-pure $i$-banner complexes from Example \ref{nonpure-example}
that are constructed by gluing a $(d-1)$-simplex to the $(i-2)$-skeleton of a 
simplex of dimension at least $i-1$ answer Question \ref{quest2} in the 
affirmative.  Indeed, each of these complexes is $(i-1,1^{d-i-1})$-balanced.

Another intriguing direction is to study face numbers of $i$-banner
homology spheres. The celebrated Charney--Davis conjecture \cite{ChD}
posits that if $\Delta$ is a $(2e-1)$-dimensional flag complex and if,
in addition, $\Delta$ is  a homology sphere, then
$(-1)^{e} \sum_{j=0}^{2e} (-1)^j h_{j}(\Delta)\geq 0$.
Gal's conjecture \cite{Gal} generalizes the Charney--Davis conjecture for
flag spheres. It asserts that all of the coefficients of a certain
$\gamma$-polynomial associated with a sphere are non-negative if the sphere
is flag. Gal's conjecture was further generalized to a series of conjectures
in \cite{NePe}.
It would be extremely interesting to find appropriate
analogs of these conjectures for $i$-banner spheres.


\end{document}